\documentclass[12pt,a4paper,reqno]{amsart}

\usepackage{graphicx}
\usepackage{hyperref}
\usepackage[margin=3cm]{geometry}
\usepackage{mathrsfs}
\usepackage{overpic}
\usepackage[utf8]{inputenc}
\usepackage{amsmath}
\usepackage{amsfonts}
\usepackage{amssymb}
\usepackage{enumerate}
\usepackage{subcaption}

\newcommand{\Z}{\ensuremath{\mathbb{Z}}}

\newcommand{\R}{\ensuremath{\mathbb{R}}}
\newcommand{\C}{\ensuremath{\mathbb{C}}}
\newcommand{\T}{\ensuremath{\mathbb{T}}}
\renewcommand{\epsilon}{\varepsilon}
\renewcommand{\geq}{\geqslant}
\renewcommand{\leq}{\leqslant}

\newtheorem{thm}{Theorem}[section]

\newtheorem{cor}[thm]{Corollary}
\newtheorem{lem}[thm]{Lemma}
\newtheorem{prop}[thm]{Proposition}

\title[Resonance of bounded isochronous oscillators]
{Resonance of bounded isochronous oscillators}

\author[D.~Rojas]{David Rojas}
\address{Departament d' Informàtica, Matemàtica Aplicada i Estadística, Universitat de Girona, 17003  Girona, Spain}
\email{david.rojas@udg.edu}



\begin{document}

\begin{abstract}
An oscillator is called isochronous if all motions have a common period. When the system is forced by a time-dependent perturbation with the same period the phenomenon of resonance may appear. We give a sufficient condition on the perturbation in order that resonance occurs when the period annulus of the isochronous oscillator is bounded. In this context, resonance means that all solutions escape from the period annulus.

\end{abstract}

\maketitle

\section{Introduction} \label{sec:intro}

An oscillator with equation
\begin{equation}\label{iso}
\ddot x + V'(x) = 0
\end{equation}
is called isochronous if it only has one equilibrium point and all solutions in a neighbourhood are periodic with a fixed period, lets say $T=2\pi$. When a small periodic perturbation with the same period as the isochronous center is added to the force, the phenomenon of resonance may occur. That is, all solutions of the non-autonomous equation 
\begin{equation}\label{per}
\ddot x + V'(x) = \epsilon p(t)
\end{equation}
are unbounded for $\epsilon\neq 0$ small. In the recent years the classical theory of resonance has been extended from the linear oscillator to nonlinear isochronous oscillators. We refer for instance \cite{Bonheure,Ortega2002} for the construction of forcings and \cite{OrtRoj2019} for sufficient conditions to produce resonance. Also~\cite{AloOrt98,Liu99} where the authors treated the specific case of the asymmetric oscillator.

Until now the oscillators treated have been defined over the whole real line or they have had an asymptote. In the first case the potential generates a global center in $\R^2$ whereas in the second case the center is defined in a semi-plane. In both situations the center is global. That is, all solutions are well-defined and $2\pi$-periodic.  
In this work we treat the case when the isochronous oscillator is bounded (see Figure~\ref{fig1}.) In general, a planar bounded center is usually confined inside a homoclinic or heteroclinic connection. In particular, an equilibrium of the equation~\eqref{iso} can be found at the outer boundary of the period annulus. Clearly, this situation is incompatible with isochronicity. However, bounded isochronous centers can be constructed using a singular potential function. In order to differ from the second case mentioned above, the singularity must be integrable. That is, equation~\eqref{iso} is singular but the Hamiltonian $H(x,\dot{x})=\frac{1}{2}\dot{x}^2+V(x)$ is not (see \cite{Lazer,Pedro} and references therein.) If this is the case, equation~\eqref{iso} is said to have a weak singularity. 
Our main result shows that the resonance condition given in~\cite{OrtRoj2019} for global centers also produces resonance for the bounded isochronous oscillator. The main difference is that in the present situation resonance is understood as the escape from the period annulus. More precisely, for a given non-empty compact subset $\mathcal{K}$ inside the bounded period annulus almost all solutions inside $\mathcal{K}$ leave the compact subset at some time if $\epsilon\neq 0$ is small enough. We point out the difference between the proofs in~\cite{OrtRoj2019} and the one in the present paper, since the Second Massera's theorem used does not apply in the bounded scenario.

At first glance bounded isochronous oscillators seem rare, but they are not. The characterization of isochronous potentials given by Urabe in~\cite{Urabe61,Urabe62} shows, roughly speaking, that there are as many bounded isochronous as odd functions $S\in C(I)\cap C^1(I\setminus\{0\})$, where $I=(\alpha,\beta)$, $\alpha<0<\beta$, satisfying $S(\alpha)+1=0$. Among all of them, the simplest one is constructed by taking $S(X)=X$, which gives the bounded potential isochronous center
\begin{equation}\label{iso_example}
\ddot{x}+1-\frac{1}{\sqrt{2x+1}}=0, \ x\in(-\tfrac{1}{2},\tfrac{3}{2}).
\end{equation}
In general, to check the resonance condition may be difficult. In the case of the previous potential, we have integrated the equation in terms of an implicit identity, which keeps close similarity to the Kepler equation, and obtained the resonance condition for periodic lineal perturbations.

The rest of the paper is organized as follows. The precise statements of the results are found in Section~\ref{sec:2}.  In Section~\ref{sec:3} we present the proof of the sufficiency of the resonance condition for bounded isochronous potentials. Section~\ref{sec:4} is dedicated to illustrate the construction of such isochronous centers using the theory of Urabe, and in Section~\ref{sec:5} we study in detail the particular system~\eqref{iso_example}. The paper is finished with some comments about the behaviour of the solutions near the boundary of the period annulus.

\section{Statement of the result}\label{sec:2}
Let us consider a potential $V\in C^2(I)$ defined in an interval $I=(\alpha,\beta)$ with $-\infty<\alpha<0<\beta<+\infty$ satisfying
\[
V(0)=V'(0)=0 \text{ and } xV'(x)>0 \text{ if }x\neq 0.
\]
In addition we assume that 
\[
\lim_{x\rightarrow\alpha^+}V(x)=\lim_{x\rightarrow\beta^-}V(x)=\overline{V}<+\infty, \text{ and }\lim_{x\rightarrow\alpha^+}V'(x)=-\infty.
\]
Under the hypothesis above, the equation~\eqref{iso} has a center at the origin with a bounded period annulus, namely $\mathscr{P}$ (see Figure~\ref{fig1}.) Solutions with initial conditions outside the period annulus are not globally defined. Indeed, those solutions reach the singularity $x=\alpha$ in finite time.

\begin{figure}
    \centering
    \includegraphics[scale=1]{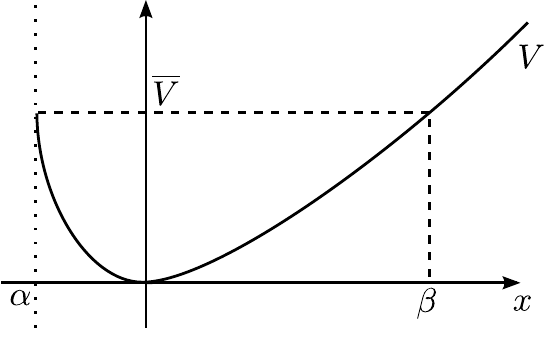} \hspace*{1cm}
    \includegraphics[scale=1]{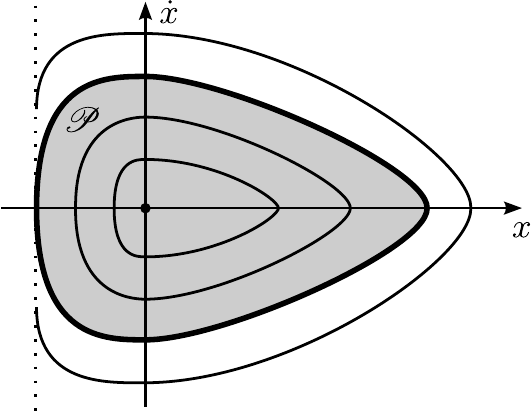}
    \caption{\label{fig1}On the left, potential function with a weak singularity at $x=\alpha$. On the right, the phase portrait of the potential system. The grey region corresponds to the period annulus. Its outer boundary is emphasized in bold.}
\end{figure}

In the forthcoming we assume that all globally-defined solutions of equation~\eqref{iso} (that is, those inside the period annulus) are $2\pi$-periodic. For every $\epsilon\in\R$, let us denote by $\varphi_{\epsilon}(t,\textbf{x})$ the solution of the first order system associated to the perturbed equation~\eqref{per} with initial condition $\textbf{x}=(x,\dot x)\in\mathscr P$. Throughout this paper we shall use the notation 
\[
\varphi_{\epsilon}(t,\mathcal U)\!:=\{\varphi_{\epsilon}(t,\textbf{x}): \textbf{x}\in\mathcal U\}
\] 
to refer the set of solutions of equation~\eqref{per} with initial conditions in the subset $\mathcal U\subset \mathscr P$. 

Let $\mathcal C=(\R/2\pi\Z)\times[0,+\infty)$ be a cylinder with coordinates $(\theta,r)$ and denote by $\phi(t,r)$ the solution of equation~\eqref{iso} with initial conditions $x(0)=r$ and $\dot{x}(0)=0$. That is, $\phi(t,r)$ is the first component of the solution $\varphi_0(t,(r,0))$ of the unperturbed equation.
The complex-valued solution of the variational equation
\[
\ddot{y}+V''(\phi(t,r))y=0,\ y(0)=1,\ \dot{y}(0)=i,
\]
is denoted by $\psi(t,r)$. We define the function
\[
\Phi_p:\mathcal C\rightarrow\C, \ \Phi_p(\theta,r)\!:=\frac{1}{2\pi}\int_0^{2\pi}p(t-\theta)\psi(t,r)dt.
\]
The resonance result for oscillators with a weak singularity states as follows. 

\begin{thm}\label{teorema}
Assume that $V$ satisfies the previous conditions and the condition
\begin{equation}\label{rc}
\inf_{\mathcal C}|\Phi_p(\theta,r)|>0
\end{equation}
holds for some $p\in L^1(\T)$. Then for every compact subset $\mathcal{K}\subset\mathscr{P}$ with non-empty interior and $d>0$ there exists $\varepsilon^*>0$ such that if $|\epsilon|<\epsilon^*$ then every open ball of initial conditions $B\subset\mathcal{K}$ with diameter $d$ satisfies $\varphi_{\epsilon}(t,B)\not\subset\mathcal{K}$ for some time $t>0$.
\end{thm}

As an example of a periodically perturbed bounded isochronous oscillator we consider the equation
\begin{equation}\label{per_example}
\ddot{x}+1-\frac{1}{\sqrt{2x+1}}=\epsilon p(t).
\end{equation}
The following result illustrates the application of the theorem to that equation when the forcing is a linear trigonometric function.

\begin{cor}\label{corollary}
Assuming that $p(t)$ is a trigonometric function of the form
\[
p(t)=a_0+a_1\cos t + b_1\sin t,
\]
the equation~\eqref{per_example} is resonant on compact subsets if
\begin{equation}\label{rc_example}
a_1^2+b_1^2>\frac{9a_0^2}{4(J_1(1)-J_2(1))^2}.
\end{equation}
\end{cor}

In the previous statement, resonance on compact subsets is understood in the terms of Theorem~\ref{teorema}. Moreover, $J_{\nu}(x)$ stands for the Bessel function of the first kind (see for instance~\cite{AbrSte65}.)

\section{Proof of Theorem~\ref{teorema}}\label{sec:3}

Before starting with the proof of the result itself, we outline the two key ideas that are behind it. The first part follows the proof of Theorem A in~\cite{OrtRoj2019} concerning the non-existence of $2\pi$-periodic solutions of equation~\eqref{per}. In the mentioned reference, the potential is defined in the whole real line whereas the potential we are concerning with reaches a singularity at $x=\alpha>-\infty$. For this reason, the result we present in Proposition~\ref{proposition} is stated in terms of compact subsets $\mathcal{K}\subset\mathscr{P}$ instead of the whole period annulus. That is, equation~\eqref{per} has no $2\pi$-periodic orbits inside any compact subset of $\mathscr{P}$ if condition~\eqref{rc} is fulfilled. The second part of the proof uses a Fixed Point theorem due to Montgomery. Let $\Delta\subset\R^2$ be a simply connected open subset with finite measure and $h:\Delta\rightarrow\Delta$ an orientation preserving homeomorphism.

\begin{thm}[Montgomery]
If $h$ is area preserving then it has a fixed point.
\end{thm}  

For a proof and further details about the previous result we refer to~\cite[Section~3.9]{Ortega2019}. The proof of Theorem~\ref{teorema} will follow by showing that the Poincaré map at time $t=2\pi$ satisfies the assumptions of Montgomery's theorem on a simply connected invariant open subset generated by the Poincaré iterates of a ball that remains inside the compact subset $\mathcal{K}$ for all time. Thus a fixed point must exists inside it. That fixed point corresponds to a $2\pi$-periodic orbit and the contradiction with the first part stated above will finish the proof. The delicate part of the proof is indeed the construction of the invariant subset.

In the forthcoming two sections we give the detailed proof of the result.

\subsection{Non-existence of \texorpdfstring{$2\pi$}{2pi}-periodic orbits on compact subsets}

Next lemma contains the key technical details to prove the proposition that follows it. It can be proved by using classical techniques of differential equations. We refer to~\cite[Section 5.1]{OrtRoj2019} for the proof.

\begin{lem}
The following two properties hold for the lineal equation
\begin{equation}\label{lin_eq}
\ddot y + a(t)y = b(t),
\end{equation}
where $a\in L^{\infty}(0,2\pi)$, $b\in L^1(0,2\pi)$ and $\|a\|_{L^{\infty}(0,2\pi)}\leq A$ for some $A>0$.
\begin{enumerate}[$(a)$]
\item Let $y(t)$ be the solution of~\eqref{lin_eq} with $b\equiv 0$ and initial conditions $y(0)=1$, $\dot y(0)=i$. Then there exists $C>0$, depending only on $A$, such that $|y(t)|\leq C$ for each $t\in[0,2\pi]$.
\item Let $y(t)$ be the solution of~\eqref{lin_eq} with initial conditions $y(0)=\dot y(0)=0$. Then $|y(t)|\leq C^2\|b\|_{L^1(0,2\pi)}$ if $t\in[0,2\pi]$.
\end{enumerate}
\end{lem}

The proof of the forthcoming proposition follows the same ideas than the ones in the proof of Theorem~A in~\cite[Section 5.1]{OrtRoj2019} with some  subtle differences. However, we decided to include the details for the sake of completeness.

\begin{prop}\label{proposition}
Under the hypothesis of Theorem~\ref{teorema}, for every compact subset $\mathcal{K}\subset\mathscr{P}$ with non-empty interior there are no $2\pi$-periodic solutions of equation~\eqref{per} inside $\mathcal K$ for $\epsilon\neq 0$ small.
\end{prop}

\begin{proof}
The compactness of the subset $\mathcal{K}$ together with the classical results of continuous dependence on parameters of differential equations ensures the existence of $\epsilon^{**}=\epsilon^{**}(\mathcal{K})>0$ in such a way all solutions of equation~\eqref{per} with $|\epsilon|<\epsilon^{**}$ and initial conditions inside $\mathcal{K}$ are well-defined for $t\in[0,2\pi]$. More concretely, they are all defined in an open interval $I$ containing $[0,2\pi]$. In order to show the result we take any sequence $\{\epsilon_n\}_{n\geq 0}$ with $\epsilon_n\rightarrow 0$ and satisfying $0<|\epsilon_n|<\epsilon^{**}$ for all $n\geq 0$ and assume, with the aim of reaching contradiction, that the equation~\eqref{per} with $\epsilon=\epsilon_n$ has a $2\pi$-periodic solution, namely $\textbf{x}_n(t)=(x_n(t),\dot x_n(t))$, contained inside $\mathcal{K}$. Let $\textbf{X}_n(t)=(X_n(t),\dot X_n(t))$ be the solution of the isochronous equation~\eqref{iso} with same initial conditions of $\textbf{x}_n(t)$ at $t=0$; that is, $X_n(0)=x_n(0)$ and $\dot X_n(0)=\dot x_n(0)$. 

First, the difference function $\textbf{y}_n(t)=\textbf{x}_n(t)-\textbf{X}_n(t)$ is a solution of the equation~\eqref{lin_eq} with
\[
a(t)=\int_0^1 V''\bigl((1-\lambda)x_n(t)+\lambda X_n(t)\bigr)d\lambda, \text{ and }\ b(t)=\epsilon_n p(t).
\]
We notice at this point that $\textbf{y}_n(t)$ is well-defined for $t\in[0,2\pi]$. Moreover, the solution $\textbf{X}_n(t)$ is contained in a compact subset $\tilde{\mathcal{K}}\subset \mathscr{P}$ which can be assumed to satisfy $\mathcal{K}\subset\tilde{\mathcal{K}}$. Thus, since the projection over the $x$-axis of $\mathscr P$ is $(\alpha,\beta)$, the convex combination $(1-\lambda)x_n(t)+\lambda X_n(t)$ is contained inside a compact interval $\tilde{I}\subset (\alpha,\beta)$ (the projection of $\tilde{\mathcal{K}}$ over the $x$-axis.) The regularity of the function $V$ on $\tilde{I}$ implies the existence of the value $A\!:=\sup_{x\in\tilde{I}}|V''(x)|.$ In consequence, we deduce from property~(b) in Lemma~\ref{lin_eq} that
\begin{equation}\label{ineq1}
\|y_n\|_{L^{\infty}(0,2\pi)}\leq C^2|\epsilon_n|\|p\|_{L^1(0,2\pi)}.
\end{equation}

Second, the function $\textbf{y}_n(t)$ can also be interpreted as a $2\pi$-periodic solution of the periodic lineal equation
\[
\ddot y + V''(X_n(t))y=\epsilon_n p(t)-q_n(t),
\]
with
\begin{equation}\label{equal1}
q_n(t)=y_n(t)\int_0^1 \bigl[ V''\bigl((1-\lambda)x_n(t)+\lambda X_n(t)\bigr)-V''\bigl(X_n(t)\bigr)\bigr]d\lambda.
\end{equation}
Every solution of the isochronous equation~\eqref{iso} is of the form $\phi(t-\theta,r)$. In particular, the function $X_n(t)=\phi(t-\theta_n,r_n)$ for some $\theta_n\in[0,2\pi]$ and $r_n\geq 0$. Let us denote by $\psi(t-\theta_n,r_n)$ the $2\pi$-periodic complex-valued solution of the homogeneous equation $\ddot y + V''(X_n(t))t=0$. The Fredholm alternative in this case implies that
\[
\epsilon_n\int_0^{2\pi} p(t)\psi(t-\theta_n,r_n)dt - \int_0^{2\pi} q_n(t)\psi(t-\theta_n,r_n)dt =0.
\]
Or, equivalently,
\[
\Phi_p(\theta_n,r_n)=\frac{1}{2\pi}\int_0^{2\pi} \frac{q_n(t)}{\epsilon_n}\psi(t-\theta_n,r_n)dt.
\]
From the uniform continuity of $V''$ in $\tilde{I}$ and inequality~\eqref{ineq1} we have that $\frac{1}{\epsilon_n}q_n(t)\rightarrow 0$ uniformly as $n\rightarrow+\infty$. In addition, from property~(a) in Lemma~\ref{lin_eq} we have that $\|\psi(\cdot,r_n)\|_{L^{\infty}(0,2\pi)}\leq C$. Consequently, $\Phi_p(\theta_n,r_n)\rightarrow 0$ uniformly as $n\rightarrow+\infty$. This last limit contradicts the resonance condition~\eqref{rc} and proves the veracity of the result.
\end{proof}

\subsection{Resonance on compact subsets}

Let us consider the Poincaré maps
\[
P_{\epsilon}:\mathcal{K}\rightarrow(\alpha,+\infty)\times\R
\]
defined by $P_{\epsilon}(\textbf{x})=\varphi_{\epsilon}(2\pi,\textbf{x})$.  We recall that there exists $\epsilon^{**}>0$ such that if $|\epsilon|<\epsilon^{**}$ then all solutions of~\eqref{per} with initial condition inside $\mathcal{K}$ are well-defined for $t\in[0,2\pi]$. 

In order to prove Theorem~\ref{teorema} we proceed by contradiction and assume that for each $|\epsilon|<\epsilon^{**}$ there exists a ball of initial conditions $B\subset\mathcal{K}$ with diameter $d$ that remains inside $\mathcal{K}$ under the positive flow of equation~\eqref{per}. That is, $\varphi_{\epsilon}(t,B)\subset\mathcal{K}$ for all $t>0$. In particular, $P_{\epsilon}^n(B)\subset \mathcal{K}$ for all $n\geq 1$. Since the Poincaré map is uniformly continuous on the compact $K$ and the diameter of the ball is fixed, there exists $0<\epsilon^*<\epsilon^{**}$ such that if $|\epsilon|<\epsilon^*$ two successive iterates of $B$ intersect for all $n\geq 0$. In particular, the union of all the Poincaré iterates of $B$,
\[
\mathcal B_{\epsilon}\!:=\cup_{n\geq 0}P_{\epsilon}^n(B),
\]
is an open connected subset of $\mathcal{K}$ which is positively invariant under the map $P_{\epsilon}$. That is, $P_{\epsilon}(\mathcal B_{\epsilon})\subset \mathcal B_{\epsilon}$. Intuitively, positively invariant open sets with good topological properties are invariant under a measure-preserving flow. In~\cite[Proposition~3]{Ortega2006} the author shows that this is true when taking the set $\text{int}(\overline{\mathcal{B}_{\epsilon}})$. That is $P_{\epsilon}(\text{int}(\overline{\mathcal{B}_{\epsilon}}))=\text{int}(\overline{\mathcal{B}_{\epsilon}})$. The open set $\text{int}(\overline{\mathcal{B}_{\epsilon}})$ is then connected and invariant under $P_{\epsilon}$ but it is not necessarily simply connected. To overcome this obstruction, we can produce a new open simply connected subset $\hat{\mathcal{B}_{\epsilon}}$ formed by $\text{int}(\overline{\mathcal{B}_{\epsilon}})$ and its possible holes filled. We refer to \cite[Section~4.6]{Ortega2019} for further details of this construction.

Then $\hat{\mathcal{B}_{\epsilon}}$ is an open simply connected invariant subset under the map $P_{\epsilon}$. Moreover, its measure is finite since $\hat{\mathcal{B}_{\epsilon}}\subset K$. At this point we use Montgomery's theorem to deduce the existence of a fixed point of $P_{\epsilon}$ in $\hat{\mathcal{B}_{\epsilon}}$. Clearly, such fixed point corresponds to a $2\pi$-periodic orbit of equation~\eqref{per}  for each $|\epsilon|<\epsilon^*$ contained in $K$. The existence of such periodic orbits contradicts the statement of Proposition~\ref{proposition} and so proves the veracity of Theorem~\ref{teorema}.

\section{Construction of bounded isochronous centers}\label{sec:4}

Isochronous potentials were locally characterized by Urabe in~\cite{Urabe61,Urabe62}. The term local means that the description was valid in a neighbourhood of the center. In~\cite{OrtRoj2018} the authors showed that Urabe's theory is extendible to the whole period annulus. In this section we illustrate how isochronous centers with bounded period annulus may be constructed using Urabe's characterization. In particular, we show that such oscillators are not strange in the family of isochronous potentials.

An isochronous potential oscillator with bounded period annulus is obtained if the corresponding potential function has an integrable singularity (or weak singularity.) Let us consider a potential $V\in C^2(I)$ defined in an interval $I=(\alpha,\beta)$ satisfying the assumptions in Section~\ref{sec:2}. That is, $V$ has an integrable singularity at $x=\alpha$ and defines a center with bounded period annulus. The Urabe class $\mathcal{U}(I)$ is defined as the set of functions $u:I\rightarrow\R$ satisfying $u(0)=0$, $u\in C(I)$ and $v\in C^1(I)$ where $v(x)=xu(x)$. Urabe's theory states that all solutions of the Cauchy problem
\[
\ddot{x}+V'(x)=0,\ x(0)=x_0,\ \dot x(0)=0,
\]
with initial conditions $x_0\in I$, $x_0\neq 0$, are periodic of a fixed minimal period $T$ if and only if there exists an odd function $S\in \mathcal{U}(J)$ with $J=(-(2\overline{V})^{\frac{1}{2}},+(2\overline{V})^{\frac{1}{2}})$ and $|S(X)|<1$ if $X\in J$, such that the solution $X(x)$ of
\begin{equation}\label{urabe}
\frac{dX}{dx}=\frac{2\pi}{T}\frac{1}{1+S(X)},\ X(0)=0,
\end{equation}
is defined on the interval $I$ and it satisfies $V(x)=\frac{1}{2}X(x)^2$. Here $\overline{V}$ is the maximum energy level of the center.

From the previous characterization, a potential function satisfying the hypothesis in Section~\ref{sec:2} can be constructed by choosing a function $S\in\mathcal{U}(I)$ satisfying $S(A)+1=0$ for some $A<0$ (in such a case, $A=-(2\overline{V})^{\frac{1}{2}}$.) Indeed, deriving the relation between $V$ and $X$ and using the differential equation~\eqref{urabe} we have
\[
V'(x)=\frac{2\pi}{T}\frac{X(x)}{1+S(X(x))}.
\]

The rest of this section is dedicated to show two examples of bounded isochronous potentials constructed using the previous characterization by Urabe. In any case, it is not restrictive to fix $T=2\pi$. The first example can be interpreted as the simplest isochronous center with bounded period annulus. That is, the one created by the choice $S(X)=X$. In this case, equation~\eqref{urabe} is explicitly integrable and $X(x)=\sqrt{2x+1}-1$. The resultant potential is given by
\[
V(x)=1+x-\sqrt{2x+1}.
\]
This isochronous center has appeared before in the literature. See for instance \cite{ManRojVil2017} where the authors study the period function of oscillators $\ddot{x}+x^p-x^q=0$, $p,q\in\R$. The second example is the potential given by taking $S(X)=\sin(X)$. Here equation~\eqref{urabe} is also integrable but the solution is given implicitly by the equation
\[
X(x)-\cos(X(x))=x-1.
\]
We point out the similitude of this solution with the Kepler equation
\[
M=E-e\sin(E),
\]
for the limiting case $e=1$ (see for instance~\cite{Pollard}.)

\section{Resonance condition for the explicit example}\label{sec:5}

In the previous section we have shown that the potential
\begin{equation}\label{potential_example}
V(x)=1+x-\sqrt{2x+1}
\end{equation}
generates an isochronous oscillator with bounded period annulus. It is a computation to check that the projection of the period annulus is $I=(-1/2,3/2)$ and the maximum energy level is given by $\overline{V}=1/2$. This section is devoted to the proof of Corollary~\ref{corollary}. To do so, we first need to compute the functions $\phi(t,r)$ and $\psi(t,r)$. 

On account of Urabe's theory introduced in the previous section, the change of coordinates $\{X=X(x),Y=y\}$ transforms the system
\[
\begin{cases}
\dot{x}=y,\\
\dot{y}=-V'(x),
\end{cases}
\]
into 
\[
\begin{cases}
\dot{X}=\frac{H(X)}{X}Y,\\
\dot{y}=-H(X),
\end{cases}
\]
where $H=V'\circ x$ and $x=x(X)$ is the inverse function of $X=X(x)$. It is easy to check that $X^2+Y^2$ is a first integral so periodic orbits in these new variables are circular. Thus, using polar coordinates the previous system finally writes 
\[
\begin{cases}
\dot{r}=0,\\
\dot{\theta}=-\omega(r\cos\theta),
\end{cases}
\]
where $\omega(X)=\frac{H(X)}{X}$. In the particular case of the equation with potential~\eqref{potential_example} we have $X(x)=\sqrt{2x+1}-1$, $x(X)=\frac{1}{2}X^2+X$ and $\omega(X)=\frac{1}{1+X}$, so
\begin{equation}\label{dtheta}
\dot{\theta}=-\frac{1}{1+r_0\cos\theta}.
\end{equation}
We point out that the equation above is well-defined for $r_0\in[0,1)$, which corresponds to initial conditions $x_0\in I$. For the sake of simplicity, we shall write $r=r_0$ from now on. Taking as initial condition for the angle $\theta_0=0$ (in correspondence with $\dot{x}(0)=0$) the solution of the equation above satisfies the identity
\begin{equation}\label{theta}
\theta(t,r)+r\sin(\theta(t,r))=-t.
\end{equation}
In consequence, the solution of the isochronous system~\eqref{iso} with potential function in~\eqref{potential_example} is given by
\[
\phi(t,r)=\left(\frac{1}{2}X^2+X\right)\!\!\bigl(r\cos(\theta(t,r))\bigr)=\frac{r^2}{2}\cos^2(\theta(t,r))+r\cos(\theta(t,r)).
\]
From the previous expression and identity~\eqref{theta} it is now a computation to check that
\begin{equation}\label{psi}
\psi(t,r)=\frac{\partial\phi}{\partial r}(t,r)-\frac{1}{V'(x_0)}\dot{\phi}(t,r)i=r+\cos(\theta(t,r))-(r+1)\sin(\theta(t,r))i,
\end{equation}
where $x_0=\frac{1}{2}r^2+r$. Now we are in position to prove the result.

\begin{proof}[Proof of Corollary~\ref{corollary}]
The function $\psi(t,r)$ has a Fourier expansion of the type
\[
\psi(t,r)=\sum_{m=-\infty}^{+\infty}c_m(r)e^{imt}.
\]
From identity~\eqref{theta} we have that $\theta(-t,r)=-\theta(t,r)$. Consequently, the identity
$\psi(-t,r)=\overline{\psi(t,r)}$ holds and so the functions $c_0(r)$, $c_1(r)$ and $c_{-1}(r)$ are real. We first show some bounds of these first coefficients. Let us define
\[
d_{+}(r)=\frac{1}{2}(c_1(r)+c_{-1}(r))=\frac{1}{2\pi}\int_{0}^{2\pi}\text{Re}\psi(t,r)\cos(t)dt,
\]
and
\[
d_{-}(r)=\frac{1}{2}(c_1(r)-c_{-1}(r))=\frac{1}{2\pi}\int_{0}^{2\pi}\text{Im}\psi(t,r)\sin(t)dt.
\]
From the expression in~\eqref{psi} we have that
\[
d_{+}(r)=\frac{1}{2\pi}\int_{0}^{2\pi}\bigl(r+\cos(\theta(t,r))\bigr)\cos(t)dt.
\]
Let us consider the change of variables $\theta=\theta(t,r)$. We invoke equality~\eqref{dtheta} to write
\begin{align*}
d_{+}(r)&=\frac{1}{2\pi}\int_{0}^{2\pi}\cos(\theta+r\sin\theta)(r\cos\theta+1)(r+\cos\theta)d\theta\\
&=\frac{1}{2\pi}\int_{0}^{2\pi}\sin(\theta+r\sin\theta)\sin\theta d\theta.
\end{align*}
Here in the last equality we used integration by parts. This last integral can be expressed in terms of Bessel functions of first kind (see~\cite{AbrSte65}), which allows to finally write
\[
d_{+}(r)=\frac{1}{r}J_1(r)-J_2(r).
\]
Using the series expansion of the previous functions we find that
$\lim_{r\rightarrow 0} d_+(r) = \frac{1}{2}$ and $d_+(r) > 0$ for all $r\in[0,1]$. The derivative writes
\[
\frac{d}{dr}\left(\frac{1}{r}J_1(r)-J_2(r)\right)=\frac{(2-r^2)J_1(r)}{r^2}-\frac{J_0(r)}{r},
\]
and using again series expansion it is easy to check that is zero at $r=0$ and negative for $r\in(0,1]$. Consequently, $d_+(r)$ is decreasing from $r=0$ to $r=1$ and so the following bounds hold for $r\in[0,1)$,
\begin{equation}\label{bound1}
0< J_1(1)-J_2(1)\leq d_{+}(r)\leq \frac{1}{2}.
\end{equation}
A similar argument shows that
\[
d_{-}(r)=-\frac{1}{r}J_1(r),
\]
and
\[
c_0(r)=\frac{1}{2\pi}\int_0^{2\pi}\text{Re}\psi(t,r)dt=\frac{3}{2}r.
\]
In consequence,
\begin{equation}\label{bound2}
J_1(1)\leq |d_{-}(r)|\leq \frac{1}{2}, \text{ and } 0\leq c_0(r)\leq\frac{3}{2}
\end{equation}
for $r\in[0,1)$.

Now we are in condition to prove the result. From the definition of $\Phi_p(\theta,r)$,
\[
2\pi\Phi_p(\theta,r)=a_0 c_0(r)+d_+(r)(a_1\cos(\theta)-b_1\sin(\theta))+d_-(r)(b_1\cos(\theta)+a_1\sin(\theta))i.
\]
The right-hand side of the equality above can be identified to the vector $w+DR(\theta)v$ in $\R^2$, where
\[
v=\left(\begin{matrix}
a_1\\
b_1
\end{matrix}\right), \ 
w=\left(\begin{matrix}
a_0 c_0(r)\\
0
\end{matrix}\right), 
D=\left(\begin{matrix}
d_+(r) & 0\\
0 & d_-(r)
\end{matrix}\right),
\]
and $R(\theta)$ stands for the rotation matrix in the counter-clockwise sense. We invoke inequalities~\eqref{bound1} and~\eqref{bound2} to deduce that
\[
|2\pi\Phi_p(\theta,r)|=|w+DR(\theta)v|=|D(D^{-1}w+R(\theta)v)|\geq (J_1(1)-J_2(1))|D^{-1}w+R(\theta)v|,
\]
and
\[
|D^{-1}w+R(\theta)v|\geq |v|-|D^{-1}w|=\!\sqrt{a_1^2+b_1^2}-\frac{c_0(r)|a_0|}{d_+(r)}\geq \sqrt{a_1^2+b_1^2}-\frac{3|a_0|}{2(J_1(1)-J_2(1))}.
\]
Finally, condition~\eqref{rc_example} and inequality~\eqref{bound1} ensures that
\[
|\Phi_p(\theta,r)|\geq \frac{1}{2\pi}(J_1(1)-J_2(1))\left(\sqrt{a_1^2+b_1^2}-\frac{3|a_0|}{2(J_1(1)-J_2(1))}\right)>0
\]
and so resonance occurs applying Theorem~\ref{teorema}.
\end{proof}



\section*{Acknowledgements}
The author want to thank Rafael Ortega for fruitful discussions and valuable comments during the development of the manuscript.

This work has been realized thanks to the \emph{Agencia Estatal de Investigación} and \emph{Ministerio de Ciencia, Innovación y Universidades} grants MTM2017-82348-C2-1-P and MTM2017-86795-C3-1-P.
 
\bibliographystyle{abbrv}

\end{document}